\documentclass[11pt]{amsart} 
\usepackage{amsmath, amssymb, amsthm}

\theoremstyle{plain}
\newtheorem{theorem}{Theorem}
\newtheorem{lemma}[theorem]{Lemma}
\newtheorem{proposition}[theorem]{Proposition}
\newtheorem{corollary}[theorem]{Corollary}
\newtheorem{claim}[theorem]{Claim}
\theoremstyle{definition}
\newtheorem{definition}[theorem]{Definition}
\newtheorem{example}[theorem]{Example}

\theoremstyle{remark}
\newtheorem{remark}[theorem]{Remark}

\newcommand{\Aut}{\operatorname{Aut}}

\title{A note on SG points for reduced plane curves}

\author{Aki Ikeda}
\address{Graduate School of Science and Technology, Niigata University, 8050 \\ Ikarashi 2-no-cho, Nishi-ku, Niigata 950-2181, Japan}
\email{ikeda@m.sc.niigata-u.ac.jp}

\author{Takeshi Takahashi}
\address{Faculty of Engineering, Niigata University, 8050 Ikarashi 2-no-cho, Nishi-ku, Niigata 950-2181, Japan}
\email{takeshi@eng.niigata-u.ac.jp}



\begin{document}

\begin{abstract}
In \cite{IT24}, we generalized the concept of Galois points for irreducible plane curves to the case of reduced plane curves. We also introduced the concept of simultaneous Galois points, which is an equivalent concept to Galois points, and studied their number when the irreducible components are nonsingular. In this paper, we consider the remaining cases where the irreducible components are of degree $d=2$ or $3$. For the case of $d=2$, we establish a generalized version of the theorem in our previous paper. For the case of $d=3$, we classify simultaneous Galois points into the first and second kinds. We give a necessary condition for the former and provide examples for the latter.
\end{abstract}

\maketitle

\section{Introduction}
	Let $k$ be an algebraically closed field of characteristic $0$. Let $C \subset \mathbb{P}^2$ be a reduced (not necessarily irreducible) plane curve over $k$, and write $C=\bigcup_{i=1}^n C_i$ for its decomposition into irreducible components. For a point $P \in \mathbb{P}^2$, we denote by $\pi_P$ the projection from $P$, and by $\pi_P^*: k(\mathbb{P}^1) \hookrightarrow R(C)$ the monomorphism induced by $\pi_P$, where $k(\mathbb{P}^1)$ is the function field of $\mathbb{P}^1$, and $R(C)$ is the function ring of $C$.

We generalize the concept of Galois points, originally introduced for irreducible plane curves by Hisao Yoshihara (see \cite{MY00}, \cite{Yos01}).

\begin{definition}[Ikeda and Takahashi {\cite[Definition 1.3]{IT24}}]  \label{Def of G-Galois point}
	A point $P \in \mathbb{P}^2$ is called a {\itshape Galois point} for $C$ if there exists a finite group $G$  such that the function ring $R(C)$ is a Galois $G$-algebra over $\pi_P^* k(\mathbb{P}^1)$.
\end{definition}

The concept of SG points is equivalent to that of Galois points for reduced plane curves.

\begin{definition}[{\cite[Definition 1.4]{IT24}}] \label{Def of SG point}
	A point $P \in \mathbb{P}^2$ is called a {\itshape simultaneous Galois point} (``SG point'' for short) for $C$ if every field extension $k(C_i)/\pi_P^* k(\mathbb{P}^1)$ ($i=1, \dots, n$) is Galois and $k(C_i) \simeq k(C_j)$ as $\pi_P^* k(\mathbb{P}^1)$-algebra for $1\le i, j \le n$, where $k(C_i)$ is the function field of $C_i$.
\end{definition}

\begin{remark}\label{Remark for Galois points vs SG points}
	For a point $P \in \mathbb{P}^2$, $P$ is a Galois point for $C$ if and only if $P$ is an SG point for $C$ (cf.\ \cite[Lemma~2.4 and Lemma~2.6]{IT24}).
\end{remark}

By the definition of SG points, we can immediately see the following lemma from a geometric viewpoint of SG points.

\begin{lemma}[{\cite[Lemma~2.7]{IT24}}]\label{Lemma: Geometric property of SG point}
	A point $P \in \mathbb{P}^2$ is an SG point for $C=\bigcup_{i=1}^n C_i$ if and only if the following conditions are satisfied\textup{:}
	\begin{enumerate}
		\item the field extension $k(C_i)/\pi_P^* k(\mathbb{P}^1)$ is Galois \textup{(}$i=1,\dots,n$\textup{)}\textup{;}
		\item there exists a birational map $\phi_{i,j}: C_i \dashrightarrow C_j$ such that $\pi_P \circ \phi_{i,j} = \pi_P$ \textup{(}$1 \leq i, j \leq n$\textup{)}.
	\end{enumerate}
\end{lemma}

An SG point is said to be {\itshape inner} (resp.\ {\itshape outer}) if $P \in \bigcap_{i=1}^{n} C_i$ (resp.\ $P \notin \bigcup_{i=1}^{n} C_i$). We denote by $S\Delta(C_1, \dots, C_n)$ (resp.\ $S\Delta_{\mathrm{in}}(C_1, \dots, C_n)$, $S\Delta_{\mathrm{out}}(C_1, \dots, C_n)$) the set of all SG points (resp.\ inner SG points, outer SG points).

\begin{remark}\label{Remark for the assumption of nonsingular}
	Let $C=\bigcup_{i=1}^n C_i$ be a reduced plane curve whose irreducible components are nonsingular. 
	\begin{enumerate}
	\item By \cite[Lemma~2.10]{IT24}, to study the SG points for $C$, it is sufficient to consider the inner and outer SG points, respectively.
	\item Assume that there exists an SG point for $C$. Then it follows from Lemma~\ref{Lemma: Geometric property of SG point} that for $1 \leq i<j\leq n$, there exists a birational map $\phi_{i,j}: C_i \dashrightarrow C_j$, and hence $C_i$ and $C_j$ have the same genus. Because the irreducible components are nonsingular, $C_i$ and $C_j$ have the same degree. Hence, if $\deg C_i \ne \deg C_j$ for some $i, j$, then  $S\Delta(C_1, \dots, C_n) = \emptyset$. Therefore, in the study of SG points for $C$, it is natural to assume that the irreducible components have the same degree $d$.
	\end{enumerate}
\end{remark}

In \cite{IT24}, we studied the number of SG points for reduced plane curves where the irreducible components are nonsingular.

\begin{theorem}[Ikeda and Takahashi {\cite[Theorem~1.5, Lemma~2.9, Theorem~5.5]{IT24}}] \label{Theorem IT24}
	Assume that $C_i$ \textup{(}$i=1,\ldots,n$\textup{)} is a nonsingular curve of degree $d$.
	\begin{enumerate}
		\item Assume that $n=2$. If $d=2$, then $\# S\Delta_{\mathrm{out}}(C_1, C_2) = 0, 1, 3$, or $6$.
		\item If $d \geq 3$, then $\# S\Delta_{\mathrm{out}}(C_1,\ldots,C_n)\leq 1$.
		\item If $d \geq 5$, then $\# S\Delta_{\mathrm{in}}(C_1,\ldots,C_n)\leq 1$.
		\item If $d=4$, then $\# S\Delta_{\mathrm{in}}(C_1,\ldots,C_n)\leq 2$.
		\item If $d \geq 4$, then $\# S\Delta_{\mathrm{in}}(C_1,\ldots,C_n)=0$ or $\# S\Delta_{\mathrm{out}}(C_1,\ldots,C_n)=0$.
	\end{enumerate}
\end{theorem}

\begin{remark}
	Under the assumption of Theorem~\ref{Theorem IT24}, we have the following.
	\begin{enumerate}
		\item If $d=1$, then $S\Delta_{\mathrm{in}}(C_1, \dots, C_n) = \emptyset$ and $S\Delta_{\mathrm{out}}(C_1, \dots, C_n) = \mathbb{P}^2 \setminus \bigcup_{i=1}^{n} C_i$.
		\item If $d=2$, then $S\Delta_{\mathrm{in}}(C_1, \dots, C_n) = \bigcap_{i=1}^{n} C_i$.
	\end{enumerate}
\end{remark}

The purpose of this paper is to prove the following theorem, which is a generalization of  Theorem~\ref{Theorem IT24}~(a) (see also \cite[the proof of Theorem~3.1]{IT24}).

\begin{theorem}\label{Theorem SG points quadrics}
	Assume that $n\ge 2$ and that $C_i$ \textup{(}$i=1,\ldots,n$\textup{)} is a nonsingular curve of degree $d=2$. Let $\hat{C}_i$ \textup{(}$i=1,\ldots,n$\textup{)} be the dual curve of $C_i$, and 
	\[
		\mathcal{L}(\hat{C}_1, \ldots , \hat{C}_n)
		 = \{ \overline{pq} \subset \hat{\mathbb{P}}^2 \mid p, q \in \bigcap_{i=1}^{n} \hat{C}_i,\ p \ne q \},
	\]
	where $\overline{pq}$ is the line passing through $p$ and $q$.
	Then the map 
	\[
		\Phi: \mathcal{L}(\hat{C}_1,\dots , \hat{C}_n)\to S\Delta_{\mathrm{out}}(C_1, \dots, C_n)
	\] 
	defined by $\Phi(aX+bY+cZ)=(a:b:c)$ is a bijection. In particular, 
	\begin{align*}
		\# S\Delta_{\mathrm{out}}(C_1, \ldots , C_n)	 = 
		\begin{cases}
			0, & \text{if $m=0, 1$,}\\
			1, & \text{if $m=2$,}\\
			3, & \text{if $m=3$,}\\
			6, & \text{if $m=4$,}\\
		\end{cases}
	\end{align*}
	where $m = \# \bigcap_{i=1}^{n} \hat{C}_i$. 
\end{theorem}

\begin{remark}
	Let the notation be the same as in Theorem~\ref{Theorem SG points quadrics}. Then $0 \leq m \leq 4$ because the dual curve $\hat{C}_i$ is again of degree $2$.
\end{remark}

We note that (2) in \cite[Lemma 2.12]{IT24} contains an error.  We require the additional condition that $P \notin C_1 \cup C_2$, and revise the lemma as follows:

\begin{lemma}[the revised version of  {\cite[Lemma 2.12~(2)]{IT24}}] \label{Lemma birat map extend to Proj trans}
	Let $C_1, C_2 \subset \mathbb{P}^2$ be nonsingular projective curves of the same degree $d$. Let $\phi:C_1 \rightarrow C_2$ be an isomorphism.
	If	
	\begin{enumerate} 
			\item $d \geq 4$, or
			\item $d=3$ and there exists $P \in \mathbb{P}^2 \setminus (C_1\cup C_2)$ such that $\pi_P \circ \phi = \pi_P$ as a map $C_1 \rightarrow \mathbb{P}^1$,
	\end{enumerate}
	then $\phi$ is the restriction of some projective transformation of $\mathbb{P}^2$.
\end{lemma}

In fact, we have the following counterexample to \cite[Lemma 2.12]{IT24}. Let $E: Y^2Z = X^3 - XZ^2$ be an elliptic curve, and $P=(0:0:1)$. Consider the projection $\pi_P : E \to \mathbb{P}^1$, $(X:Y:Z) \mapsto (X:Y)$, and the morphism $\phi : E \to E$, $(X:Y:Z) \mapsto (-XZ : -YZ : X^2)$. It is easy to check that $\pi_P \circ \phi = \pi_P$. Because $\phi(P) = (0:1:0) \neq P$, we see that $\phi$ cannot be extended to a projective transformation. More generally, for an elliptic curve $E$ given in Weierstrass form and a point $P \in E$ that is not a $3$-torsion point, the morphism $\sigma_P := [-1] \circ \tau_P$ satisfies (i) $\pi_P \circ \sigma_P = \pi_P$ and (ii) $\sigma_P \notin \Aut(\mathbb{P}^2)$, where $[-1]$ is the multiplication-by-$(-1)$ map and $\tau_P$ is the translation-by-$P$ map on $E$. 
	
We conclude this section by summarizing the remaining case. For a reduced plane curve $C=\bigcup_{i=1}^n C_i$ whose irreducible components are nonsingular of degree $d$, we have been studying the problem of determining the number of SG points for $C$. We solve this problem by Theorem~\ref{Theorem IT24} and Theorem~\ref{Theorem SG points quadrics}, except for the inner case of degree $d=3$ (see also Table~\ref{table:SGpoints}). In Section~\ref{section: nonsingular cubic}, we consider the  remaining case $d=3$.

\begin{table}[!ht]
	\caption{The number of SG points for $C=\bigcup_{i=1}^n C_i$}
	\label{table:SGpoints}
	\centering
	\begin{tabular}{|c|c|c|c|c|c|} \hline
   		degree $d$ & $1$ & $2$ & $3$ & $4$ & $d \geq 5$ \\ \hline
   		$\# S\Delta_{\mathrm{in}}(C_1, \ldots , C_n)$ & $0$ & $\# \bigcap_{i=1}^n C_i$ & ? & $\leq 2$ & $\leq 1$\\
   		$\# S\Delta_{\mathrm{out}}(C_1, \ldots , C_n)$ & $\infty$ & $0,1,3,6$  & $\leq 1$ & $\leq 1$ & $\leq 1$ \\ \hline
 	\end{tabular}
\end{table}

\section{Proof of Theorem~\ref{Theorem SG points quadrics}}

Throughout this section, we work under the assumptions of Theorem~\ref{Theorem SG points quadrics}. We first prove the following claim.

\begin{claim} \label{Claim representation mathcal L}
	\[
		\mathcal{L}(\hat{C}_1, \ldots , \hat{C}_n)
		= \bigcap_{1 \le i < j \le n} \mathcal{L}(\hat{C}_i, \hat{C}_j).
	\]
\end{claim}

\begin{proof}
	Let $l \in \mathcal{L}(\hat{C}_1, \ldots , \hat{C}_n)$. 
	Then there exist distinct points $p, q \in \bigcap_{i=1}^n \hat{C}_i$ such that $l = \overline{pq}$. 
	Hence $l \in \mathcal{L}(\hat{C}_i, \hat{C}_j)$ for all $i<j$, and thus
	\[
		\mathcal{L}(\hat{C}_1, \ldots , \hat{C}_n)
		\subset \bigcap_{1 \le i < j \le n} \mathcal{L}(\hat{C}_i, \hat{C}_j).
	\]

	Conversely, let $l \in \bigcap_{i<j} \mathcal{L}(\hat{C}_i, \hat{C}_j)$. 
	Because $l \in \mathcal{L}(\hat{C}_1, \hat{C}_2)$ and $\deg \hat{C}_1 = 2$, we can write $l \cap \hat{C}_1 = \{p,q\}$ with $p \ne q$. 
	For each $j \ge 2$, the condition $l \in \mathcal{L}(\hat{C}_1, \hat{C}_j)$ implies $\{p,q\} \subset \hat{C}_j$. 
	Thus $\{p,q\} \subset \bigcap_{i=1}^n \hat{C}_i$, and hence $l \in \mathcal{L}(\hat{C}_1, \ldots , \hat{C}_n)$.
\end{proof}

We next prove Theorem~\ref{Theorem SG points quadrics}. By \cite[Lemma~2.8, the proof of Theorem~3.1]{IT24} and Claim~\ref{Claim representation mathcal L}, the map $\Phi$ is a well-defined injection and
\[
	\begin{split}
		\# S\Delta_{\mathrm{out}}(C_1, \dots, C_n)
			&= \# \bigcap_{1\le i < j\le n} S\Delta_{\mathrm{out}}(C_i, C_j) \\
			&= \# \bigcap_{1\le i < j\le n} \mathcal{L}(\hat{C}_i, \hat{C}_j) \\
			&= \# \mathcal{L}(\hat{C}_1, \ldots , \hat{C}_n).
	\end{split}
\] 
Therefore, $\Phi$ is a bijection.

\section{The case of degree $d=3$} \label{section: nonsingular cubic}
The only remaining case in the problem of determining the number of SG points, when the irreducible components are nonsingular, is the inner case of degree $d=3$. In this section, we present some partial results for this case. We use Lemma~\ref{Lemma: Geometric property of SG point} without further mention.

First, we define a certain class of SG points which are convenient from a geometric viewpoint.

\begin{definition}\label{Definition: S1}
	 Let $C_{1}, C_{2} \subset \mathbb{P}^{2}$ be nonsingular cubic curves such that $C:= C_{1}\cup C_{2}$ is reduced. A point $P \in \mathbb{P}^{2}$ is called an SG point of the {\it first kind} for $C$ if there exists a birational map $\phi : C_{1} \to C_{2}$ such that (i) $\pi_{P} \circ \phi = \pi_{P}$, and (ii) $\phi$ extends to a projective transformation. We denote by $S\Delta^1(C_1, C_2)$ the set of all SG points of the first kind for $C$, and by $S\Delta^1_{\mathrm{in}}(C_1, C_2) := S\Delta_{\mathrm{in}}(C_1, C_2) \cap S\Delta^1(C_1, C_2)$ the set of all inner SG points of the first kind for $C$.
\end{definition}

\begin{lemma}\label{lemma: local intersection number}
	Let $C_{1}$ and $C_{2}$ be nonsingular cubic curves, and $P \in C_1 \cap C_2$. Let $I_P(C_1, C_2)$ be the local intersection number of $C_1$ and $C_2$ at $P$. Assume that $P \in S\Delta^1_{\mathrm{in}}(C_{1}, C_{2})$. Then $I_P(C_1, C_2) \geq 2$, that is $T_P(C_1) = T_P(C_2)$, where $T_P(C_i)$ \textup{(}$i=1, 2$\textup{)} is the tangent line to $C_i$ at $P$.
\end{lemma}

\begin{proof}
	By the assumption, there exists a projective transformation $\phi: C_1 \to C_2$ such that $\pi_{P} \circ \phi = \pi_{P}$. Then we have $\phi(P)=P$, and $T_P(C_1) = \phi(T_P(C_1)) = T_{\phi(P)}(\phi(C_1)) = T_P(C_2)$.
\end{proof}

By Lemma~\ref{lemma: local intersection number} and B\'{e}zout's theorem, we have the following:

\begin{corollary}
	Let $C_{1}$ and $C_{2}$ be nonsingular cubic curves. Then \[\# S\Delta^1_{\mathrm{in}}(C_{1}, C_{2}) \leq 4. \]
\end{corollary}

We provide a necessary condition for $S\Delta^1_{\mathrm{in}}(C_1, C_2)$.

\begin{proposition}\label{Proposition: S1}
	Let $C_{1}$ and $C_{2}$ be nonsingular cubic curves, and $P \in C_{1} \cap C_{2}$. Let $\pi_{P}$ be the projection from $P$. Then the following condition \textup{(}a\textup{)} implies \textup{(}b\textup{)}\textup{:}
	\begin{enumerate}
		\item $P \in S\Delta^1_{\mathrm{in}}(C_1, C_2)$\textup{;}
		\item $\pi _{P}(R_{1}\setminus \{ P \}) = \pi _{P}(R_{2}\setminus \{ P \})$, where $R_{i}$ \textup{(}$i=1,2$\textup{)} is the set of ramification points of $\pi _{P}|_{C_i}$.
	\end{enumerate}
\end{proposition}

\begin{proof}
	Assume that (a). Thus there exists an isomorphism $\phi :C_{1} \to C_{2}$ which can be extended to a projective transformation and satisfies $\pi_{P} \circ \phi = \pi _{P}$. For $\Theta _{i} \in C_{i}\setminus\{ P\}$ ($i=1,2$), the line passing through $P$ and $\Theta_{i}$ is the tangent line $T_{\Theta_{i}}(C_{i})$ if and only if $\Theta _{i}$ is a ramification point of $\pi _{P}|_{C_{i}}$. Then $\phi$ induces a bijection $\phi : R_{1}\setminus \{ P \} \to R_{2}\setminus \{ P \}$. Because $\pi_{P} \circ \phi = \pi _{P}$, the assertion (b) follows.
\end{proof}

\begin{remark}\label{remark: ramification-inflection}
	For a nonsingular cubic curve $C$ and $P \in C$, the point $P$ is a ramification point of $\pi _{P}|_{C}$ if and only if $P$ is an inflection point of $C$.
\end{remark}

\begin{corollary}
	Let $C_{1}$ and $C_{2}$ be nonsingular cubic curves, and $P \in C_{1} \cap C_{2}$. If $P$ is an inflection point of $C_{1}$ but not an inflection point of $C_{2}$, then $P \notin S\Delta^1_{\mathrm{in}}(C_1, C_2)$.
\end{corollary}

\begin{proof}
	Let $\pi_i = \pi_{P}|_{C_{i}}$ and $R_{i}$ be the set of ramification points of $\pi_{i}$ ($i=1, 2$).
	For $\Theta \in C_i$, $\Theta $ is a ramification point of $\pi_i$ if and only if $\# \pi_{i}^{-1}(\pi_{i}(\Theta))=1$, because of $\deg(\pi_i) = 2$. Hence the map $\pi_{i}|_{R_{i}}: R_{i} \to \mathbb{P}^{1}$ is an injection, and we have $\# \pi_{i} (R_{i}\setminus \{ P \}) = \#(R_{i} \setminus \{ P \})$. Because  
	\[
		\#(R_{i} \setminus \{ P \}) = 
		\begin{cases}
			3, & \text{if $P$ is an inflection point of $C_{i}$,}\\
			4, & \text{otherwise}
		\end{cases}
	\]
	by Remark~\ref{remark: ramification-inflection}, it follows from Proposition~\ref{Proposition: S1} that the assertion holds. 
\end{proof}

The condition (b) of Proposition~\ref{Proposition: S1} is easy to verify by using the following lemma.
\begin{lemma} \label{lemma: Theta}
	Let $C$ be a nonsingular cubic curve defined by a homogeneous polynomial $F$, and $P=(a:b:c) \in C$. For $\Theta \in C \setminus \{ P \}$, the point $\Theta$ is a ramification point of the projection $\pi _{P}|_{C}$ from $P$ if and only if
	\[
		a \, \frac{\partial F}{\partial X}(\Theta)+ b \, \frac{\partial F}{\partial Y}(\Theta)+ c \, \frac{\partial F}{\partial Z}(\Theta) =0.
	\]
\end{lemma}

\begin{proof}
	For $\Theta \in C \setminus \{ P \}$, the point $\Theta$ is a ramification point of $\pi _{P}|_{C}$ if and only if $T_{\Theta}(C) = \overline{\Theta P}$, where $T_{\Theta}(C)$ is the tangent line to $C$ at $\Theta$, and $\overline{\Theta P}$ is the line passing through points $\Theta$ and $P$. Then the assertion follows from
	\[
		T_{\Theta}(C) : \frac{\partial F}{\partial X}(\Theta)X+ \frac{\partial F}{\partial Y}(\Theta)Y+  \, \frac{\partial F}{\partial Z}(\Theta)Z =0.
	\]
\end{proof}

By applying Proposition~\ref{Proposition: S1} to elliptic curves whose coefficients depend on two parameters, we obtain reduced plane curves with a single SG point of the first kind:

\begin{corollary}\label{Corollary (s, t)}
	Let $E_1$ and $E_2$ be elliptic curves defined by
	$E_{1}: Y^2Z = F(X, Z)$, $E_{2}: (Y-sX-tZ)^2Z = F(X, Z)$,
	where
	\[
		F(X, Z) = X^3 + \frac{s^2}{4}X^2Z + \frac{st}{2}XZ^2 + \frac{t^2}{4}Z^3
	\]
	with $(s,t) \in k^2 \setminus \{(0,0)\}$, $t \neq 0$, and $t \neq s^3/27$. Then $S\Delta^1_{\mathrm{in}}(E_{1}, E_{2}) = \{ O \}$.
\end{corollary}

\begin{proof}
	Note that the discriminant of $E_{1}$ is $t^3(s^3-27t)$. By direct calculation, it is easy to see that $E_{1}\cap E_{2} = \{ O:=(0:1:0), P_{t}:=(0:t/2:1) \}$.
	Because the projective transformation
	\[
		\phi _{s, t} :=
		\begin{pmatrix}
			1 & 0 & 0 \\
			s & 1 & t \\
			0 & 0 & 1
		\end{pmatrix}
	\]
	satisfies $\pi _{O} \circ \phi _{s, t} = \pi _{O}$, we have $O \in S\Delta^1_{\mathrm{in}}(E_{1}, E_{2})$.
	We next consider the point $P_t$. Let $R_i$ ($i=1, 2$) be the set of ramification points of $\pi _{P_{t}}|_{E_i}$.
	Using Lemma~\ref{lemma: Theta}, we obtain that
	\[
		\begin{aligned}
			R_{1} &= \{ P_{t} \} \cup \{ (\alpha: -\frac{s}{2}\alpha -\frac{3t}{2} : 1) \mid h(\alpha) =0 \}, \\
			R_{2} &= \{ P_{t} \} \cup \{ (\alpha: \frac{3s}{2}\alpha +\frac{5t}{2} : 1) \mid h(\alpha) =0 \},
		\end{aligned}
	\]
	where $h(x) = x^3-stx-2t^2$. (Note that the discriminant of $h(x)$ is $4t^3(s^3-27t) \neq 0$, so $h(x)$ has three distinct roots.) It follows from $t \neq 0$ that $\pi _{P_{t}} (R_{1}\setminus \{ P_{t} \}) \neq \pi _{P_{t}}(R_{2}\setminus \{ P_{t} \})$. Hence, by Proposition~\ref{Proposition: S1}, we have $P_{t} \notin S\Delta^1_{\mathrm{in}}(E_{1}, E_{2})$.
\end{proof}

\begin{remark}
	The elliptic curves $E_1$ and $E_2$ in Corollary~\ref{Corollary (s, t)} have the same tangent line
	\[
		T_{P_{t}} (E_i):  \frac{st}{2}X - tY + \frac{t^2}{2}Z = 0
	\]
	at $P_{t}$ ($i=1, 2$). Thus we cannot apply Lemma~\ref{lemma: local intersection number} to deduce that $P_{t} \notin S\Delta^1_{\mathrm{in}}(E_{1}, E_{2})$.
\end{remark}

For elliptic curves defined in Weierstrass form, we have the following.

\begin{proposition}
	Let $E_{1}$ and $E_{2}$ be elliptic curves given in Weierstrass form.
	Let $O=(0:1:0)$ and $\pi_{O}$ be the projection from $O$. For $n \in \mathbb{N}$, denote by $E_{i}[n]$ \textup{(}$i=1,2$\textup{)} the $n$-torsion subgroup of $E_{i}$. Then each of the following conditions implies the next one\textup{:}
	\begin{enumerate}
		\item $O \in S\Delta^1_{\mathrm{in}}(E_1, E_2)$\textup{;}
		\item $\pi _{O}(E_{1}[n]) = \pi _{O}(E_{2}[n])$ for all $n \in \mathbb{N}$\textup{;}
		\item $\pi _{O}(E_{1}[2]) = \pi _{O}(E_{2}[2])$, that is the $2$-division polynomials of $E_{1}$ and $E_{2}$ coincide\textup{;}
		\item $O \in S\Delta_{\mathrm{in}}(E_1, E_2)$.
	\end{enumerate}
\end{proposition}

\begin{proof}
	Assume that (a). Then there exists a projective transformation $\phi : E_{1}\to E_{2}$ such that $\pi_{O} \circ \phi = \pi_{O}$ and $\phi (O)=O$. Thus $\phi$ is a bijective isogeny, and (b) holds. The implication (b) $\Rightarrow$ (c) is obvious.
	Next, we assume that (c). Because the ramification points of $\pi_{O}|_{E_{i}}$ ($i=1,2$) are precisely the points in $E_{i}[2]$ and $\pi_{O}|_{E_i} ( O) = (0:1)$,  the branch loci of $\pi_{O}|_{E_{1}}$ and $\pi_{O}|_{E_{2}}$ coincide. Hence the discriminants of the quadratic extensions $k(E_{1})/\pi^{\ast}_{O}k(\mathbb{P}^{1})$ and $k(E_{2})/\pi^{\ast}_{O}k(\mathbb{P}^{1})$ are equal up to multiplication by an element of $k^{\times}$. Thus $k(E_{1})$ and $k(E_{2})$ are isomorphic over $\pi^{\ast}_{O}k(\mathbb{P}^{1})$. 
\end{proof}

Next, we consider SG points which are not of the first kind.

\begin{definition}\label{Definition: S2}
	 Let $C_{1}, C_{2} \subset \mathbb{P}^{2}$ be nonsingular cubic curves such that $C:= C_{1}\cup C_{2}$ is reduced. A point $P \in \mathbb{P}^{2}$ is called an SG point of the {\it second kind} for $C$ if $P$ is an SG point but is not of the first kind. We denote by $S\Delta^2_{\mathrm{in}}(C_1, C_2) := S\Delta_{\mathrm{in}}(C_1, C_2) \setminus S\Delta^1_{\mathrm{in}}(C_1, C_2)$ the set of all inner SG points of the second kind for $C$.
\end{definition}

\begin{example} \label{Example 2nd kind}
	Consider the nonsingular cubic curves $C_1: YZ^2 = X^3 + Y^3$, $C_2: XZ^2 = a^{2}X^2Y + 2aY^2Z$ ($a \in k^{\times}$). Then $P=(0:0:1) \in S\Delta^2_{\mathrm{in}}(C_1, C_2)$. Indeed, the isomorphism $\phi : C_1 \rightarrow C_2, (X:Y:Z) \mapsto (X(-Y+Z):\, Y(-Y+Z):\, aX^2)$ satisfies $\pi_P \circ \phi = \pi_P$, and hence $P$ is an SG point for $C_1 \cup C_2$. Because $I_P(C_1, C_2) = 1$, it follows from Lemma~\ref{lemma: local intersection number} that $P$ is not of the first kind.
\end{example}

The following proposition is useful for constructing reduced plane curves with an SG point of the second kind.

\begin{proposition}\label{proposition: a SG point given by complete linear systems}
	Let $E$ be an elliptic curve given in Weierstrass form. Assume that  $P=(0:0:1)$ is a $3$-torsion point of $E$. Let $O=(0:1:0)$ and $Q = [2]P$, where $[2]: E \to E$ is the multiplication-by-$2$ map. Let $f \in k(E)$ such that $\mathop{\mathrm{div}}\nolimits (f) =O+Q+2P$, where $\mathop{\mathrm{div}}\nolimits (f)$ is the associated divisor with $f$. Then $\Phi_{1} : E \to \mathbb{P}^{2}, (X: Y: Z)\mapsto (1: Z/X: f(X, Y, Z))$ and $\Phi_{2} : E \to \mathbb{P}^{2}, (X: Y: Z)\mapsto (1: Z/X: f(X, Y, Z)^{-1})$ are birational embeddings. Furthermore, $\Phi_1(P) = \Phi_2(O)=P$.
\end{proposition}

\begin{proof}
	Because $P=(0:1:0)$ is a $3$-torsion point of $E$ and $Q=[2]P$, we obtain that $\pi_{O}^{-1}\left( (0:1) \right) = \{ P, Q \}$. Hence $\mathop{\mathrm{div}}\nolimits (x) = P + Q -2O$. Let $D_1$ and $D_2$ be divisors on $E$ defined by $D_1=2P+Q$ and $D_2=P+Q+O$. Then the map $\Phi_i$ ($i=1, 2$) is a rational map associated with the complete linear system $|D_i|$, and hence $\Phi_i$ is a birational embedding. We have $\Phi_1(P) =\Phi_2(O) =P$ by considering the zeros and poles of the coordinate functions of $\Phi_i$.
\end{proof}

\begin{example} \label{Example2 2nd kind}
	Consider the elliptic curve $E: Y^2Z+aXYZ+bYZ^2=X^3$, where $a, b \in k$ with $b(a^3-27b) \neq 0$. Let $P=(0:0:1)$ and $Q=(0:-b:1)$. The point $P$ is a $3$-torsion point and $Q=[2]P$. Let $f=X/Y$. Because  $\mathop{\mathrm{div}}\nolimits (f) = O+Q+2P$, it follows from Proposition~\ref{proposition: a SG point given by complete linear systems} that $\Phi_1: E \to \mathbb{P}^{2}, (X:Y:Z) \mapsto (XY: YZ: X^2)$ and $\Phi_2: E \to \mathbb{P}^{2}, (X:Y:Z) \mapsto (X:Z:Y)$ are birational embeddings. Let $C_i = \Phi_i(E)$ ($i=1, 2$). Then $C_1: X^2Y+aXYZ + bY^2Z = XZ^2$ and $C_2: YZ^2+aXYZ + bY^2Z =X^3$. Because $\pi_{P} \circ \Phi_1 = \pi_{P} \circ \Phi_2$, the isomorphism $\phi:=\Phi_2 \circ \Phi_1^{-1}: C_1 \to C_2, (X:Y:Z) \mapsto (XZ:\, YZ:\, X^2)$ satisfies that $\pi_{P} \circ \phi = \pi_{P}$. Hence $P \in S\Delta_{\mathrm{in}}(C_1, C_2)$. The local intersection number $I_P(C_1, C_2) = 1$ implies that $P$ is an SG point of the second kind for $C_1 \cup C_2$.
\end{example}

\begin{remark}
	The problem of determining the number of SG points for reduced plane curves is of interest because it relates the defining equations of the irreducible components and thus reflects geometric properties of the curve. For the remaining case of $d=3$ (cf.~Table~\ref{table:SGpoints}), the maximum number of inner SG points of the first kind for $C=C_1 \cup C_2$ remains unknown. Moreover, it is not known whether there exists a reduced plane  curve admitting at least two inner SG points. For the SG points of the second kind, it is natural to ask how they can be characterized geometrically.
\end{remark}

\section*{Acknowledgments}
We would like to express our sincere gratitude to Prof.~Akinari Hoshi for his encouragement.  
This work was supported by JST SPRING Grant Number JPMJSP2121 and JSPS KAKENHI Grant Number JP25K06930.


\end{document}